\documentclass[11pt,a4paper]{article}

\DeclareMathAlphabet{\mathcal}{OMS}{cmsy}{m}{n}
\usepackage[top=25mm,bottom=35mm,left=20mm,right=20mm]{geometry} 

\usepackage{url}
\usepackage{graphicx}
\usepackage{mathptmx}     
\usepackage{mathtools}
\mathtoolsset{showonlyrefs=true}
\usepackage{color}
\usepackage{doi}
\usepackage{amsmath}
\usepackage{amsfonts}
\usepackage{amssymb}
\usepackage{amsthm}
\usepackage{bm}
\newcommand{\bx}{\bm{x}}
\newcommand{\by}{\bm{y}}
\newcommand{\bb}{\bm{b}}
\newcommand{\bc}{\bm{c}}
\newcommand{\sor}{\mathrm{SOR}}
\newcommand{\bbR}{\mathbb{R}}
\newcommand{\rmd}{\mathrm{d}}
\newcommand{\dg}{\nabla _\rmd}

\DeclareMathOperator*{\diag}{diag}

\theoremstyle{definition}
\newtheorem{proposition}{Proposition}
\newtheorem{definition}{Definition}
\newtheorem{theorem}{Theorem}

\usepackage{tikz}
\usetikzlibrary{arrows,positioning,plotmarks,external,patterns,angles,
decorations.pathmorphing,backgrounds,fit,shapes,graphs,calc}

\begin{document}
\title{On the equivalence between SOR-type methods for linear systems and 
discrete gradient methods for gradient systems}
\author{
Yuto Miyatake\thanks{Department of Applied Physics, 
              Graduate School of Engineering, 
              Nagoya University,  Furo-cho, Chikusa-ku, 
              464-8603 Nagoya, Japan, 
\href{mailto:miyatake@na.nuap.nagoya-u.ac.jp}{miyatake@na.nuap.nagoya-u.ac.jp}} ,\
Tomohiro Sogabe\thanks{Department of Applied Physics, 
              Graduate School of Engineering, 
              Nagoya University,  Furo-cho, Chikusa-ku, 
              464-8603 Nagoya, Japan, 
\href{mailto:sogabe@na.nuap.nagoya-u.ac.jp}{sogabe@na.nuap.nagoya-u.ac.jp}}
\ and 
Shao-Liang Zhang\thanks{Department of Applied Physics, 
              Graduate School of Engineering, 
              Nagoya University,  Furo-cho, Chikusa-ku, 
              464-8603 Nagoya, Japan, 
\href{mailto:zhang@na.nuap.nagoya-u.ac.jp}{zhang@na.nuap.nagoya-u.ac.jp}}
}

\maketitle

\begin{abstract}
The iterative nature of many discretisation methods for continuous dynamical systems
has led to the study of
the connections between iterative numerical methods in numerical linear algebra and continuous dynamical systems.
Certain researchers have used the explicit Euler method
to understand this connection, but this method has its limitation.
In this study, we present a new connection between
successive over-relaxation (SOR)-type methods and gradient systems;
this connection is based on discrete gradient methods.
The focus of the discussion is the equivalence between SOR-type methods and
 discrete gradient methods applied to gradient systems.
The discussion leads to new interpretations for SOR-type methods.
For example,
we found a new way to derive
these methods; these methods monotonically decrease a certain quadratic function and
obtain a new interpretation of the relaxation parameter.
We also obtained a new discrete gradient while studying the new connection.
\end{abstract}

\section{Introduction}
\label{sec1}

Iterative numerical methods in numerical linear algebra
are widely used to solve 
important mathematical problems.
Many iterative numerical methods, particularly one-step processes, are formulated as
\begin{align}\label{eq:cds}
x^{(k+1)} = T (x^{(k)}), \quad k = 0,1,2,\dots 
\end{align}
with an initial vector $x^{(0)}=x_0 \in V$. Here $V$ is an appropriate vector space,
and $T:V\to V$ is a predetermined map.
Connections between iterative numerical methods in numerical linear algebra 
and continuous dynamical systems have been studied
(since 1970s)
because of the iterative nature of many discretisation methods
used for continuous dynamical systems.
These continuous dynamical systems are represented by 
ordinary differential equations (ODEs) of the form
\begin{align}\label{eq:dds}
\frac{\rmd}{\rmd t} x(t) = \tau (x(t)), \quad t>0
\end{align}
with the initial vector $x(0) = x_0\in V$, where $\tau: V\to V$.
For example, the explicit Euler scheme with unit stepsize $h=1$ 
applied to \eqref{eq:dds} is
\begin{align*}
x^{(k+1)} = x^{(k)} + \tau (x^{(k)}).
\end{align*}
Therefore, \eqref{eq:cds} and \eqref{eq:dds} can be linked by the following equality:
\begin{align*}
T (x^{(k)}) = x^{(k)} + \tau (x^{(k)}).
\end{align*}
In fact, such connections to continuous dynamical systems 
have been studied for many iterative numerical methods,
such as stationary iterative methods for systems of linear equations~\cite{ch88,ch08},
Newton's methods for systems of nonlinear equations~\cite{ag80,sm76},
QR algorithms for eigenvalue problems~\cite{dn83,sy81} and so on.
Additional details can be obtained from~\cite{ch88,ch08}.
Studying the connections
provides a better understanding of the iterative numerical methods. 
This improved understanding can help develop better iterative numerical methods
by using state-of-the-art numerical methods for ODEs
and by devising ODEs from the viewpoint of continuous dynamical systems.

In the case of stationary iterative methods for solving linear systems,
little has been understood about the connection to continuous dynamical systems.
Let us consider linear systems of the form
\begin{align}
A \bx = \bb 
\end{align}
for a given nonsingular matrix $A\in\bbR^{n\times n}$ and vector $\bb\in\bbR^n$
(hereinafter bold notation will be used to denote vectors).
In the discrete setting,
by introducing the splitting $A = M-N$, 
stationary iterative methods for solving linear systems are formulated as
\begin{align} \label{sim} 
\bx ^{(k+1)} = G \bx ^{(k)} + \bc, \quad k = 0,1,2,\dots ,
\end{align}
where $G = M^{-1}N$ and $\bc = M^{-1}\bb$.
It is clear from the formulation that $A^{-1} \bb$ is the unique equilibrium point of \eqref{sim}.
In the continuous setting,
a system of ODEs of the form
\begin{align} \label{eq:lode}
\frac{\rmd}{\rmd t} \bx (t) = -P ( A\bx(t) - \bb), \quad t>0 
\end{align}
with a nonsingular matrix $P\in \bbR^{n\times n}$
has the unique equilibrium point $A^{-1} \bb$.
The explicit Euler scheme applied to \eqref{eq:lode} with unit stepsize $h=1$
is written as
\begin{align*}
\bx^{(k+1)} = (I-PA) \bx ^{(k)} + P\bb, \quad k = 0,1,2,\dots.
\end{align*}
Therefore,
\eqref{sim} and 
\eqref{eq:lode} can be connected by the relationship
\begin{align} \label{s:rePG}
I - PA =  G
\end{align}
via the explicit Euler method with unit stepsize.
This approach is simple and some outcomes have been discussed in~\cite{ch88,ch08};
however, 
it would be difficult to obtain an in-depth understanding because of the following reasons.

\begin{enumerate}
\item[(a)] 
For many practical and important stationary iterative methods,
such as the successive over-relaxation (SOR) method,
the iteration matrix $G$ is complicated because it often involves parameters
and matrix inverses.
For such cases, the corresponding ODEs, 
that is, the expression of the matrix $P$, are as complicated as the stationary iterative methods.
This situation makes it difficult to understand the
stationary iterative methods
in a simple way from the viewpoint of the ODEs; therefore,
the above connection rarely offers significant insights.

\item[(b)]
A fundamental difference exists between the convergence conditions of 
the stationary iterative methods \eqref{sim} and the ODEs \eqref{eq:lode};
this difference has been  explained in~\cite{ch08}.
The sequence of a stationary iterative method of the form \eqref{sim}
converges to $A^{-1}\bb$ for any initial vector 
if and only if the spectral radius of $G$ is less than one, that is, $\rho (G) <1$.
However, the flow of a system of ODEs of the form \eqref{eq:lode}
converges to $A^{-1}\bb$ for any initial vector if and only if 
all eigenvalues of $PA$ lie in the open right half plane,
because the exact flow of \eqref{eq:lode}
with the initial vector $\bx_0$ is given by 
\begin{align*}
\bx (t) = \mathrm{e}^{-PAt} (\bx_0 - A^{-1} \bb ) + A^{-1}\bb.
\end{align*}
In addition to the difference of convergence conditions, 
differences also exist between
preferable situations pertaining to quick convergence.
The matrix $G$ in the stationary iterative methods of the form \eqref{sim} 
is ideally chosen such that $\rho (G)$,
the spectral radius of $G$, is as small as possible.
However,
the matrix $P$ in ODEs of the form \eqref{eq:lode} is ideally chosen such that 
the real part of 
all eigenvalues of $PA$ is as large as possible.
These two requirements are not usually satisfied simultaneously under \eqref{s:rePG}.
\end{enumerate}

Given these consideration, our aim is to seek a new possible connection between 
the stationary iterative methods and the ODEs such that 
meaningful insights could be obtained, and 
there is no difference of convergence conditions.
In particular, we focus on SOR and related methods,\footnote{Note that although these are classical stationary iterative methods, 
they are still receiving a great deal of attention.
For recent developments, see~\cite{ng16} and~\cite{oz17}.}
and discuss their connections to certain gradient systems. 
Thus,
we restrict $A$ to be symmetric positive definite.
Let
\begin{align}
f(\bx) = \frac{1}{2} \bx^\top A \bx - \bx^\top \bb.
\end{align}
We consider the stationary iterative methods with the property $f(\bx^{(k+1)}) \leq f(\bx^{(k)})$
and linear ODEs with the property $\frac{\rmd}{\rmd t} f(\bx (t)) \leq 0$;
we make a new connection between the two classes.
The main idea is to use the discrete gradient method~\cite{go96,ia88,mq99,qm08,qt96} 
instead of the explicit Euler method
to obtain a fresh understanding of the stationary iterative methods.
In the rest of the paper, we start the discussion with the continuous case.
In Section~\ref{sec2},
we consider gradient systems as a subclass of linear ODEs and review the 
basic properties of gradient systems as preliminaries.
In Section~\ref{sec3}, we review the discrete gradient method
and show that the discrete gradient schemes applied to gradient systems
constitute a certain subclass of stationary iterative methods.
We also show that the SOR method is connected to a certain gradient system
by proving the equivalence between the SOR method and the discrete gradient method;
the consequences of this equivalence are also discussed. 
In Section~\ref{sec4}, we show that some SOR-type methods are also connected to gradient systems.
Finally, Section~\ref{sec:cr} contains concluding remarks.

\section{Preliminaries}
\label{sec2}
We first summarise the basic properties of gradient systems in the general setting
and then explain the specific gradient systems associated with symmetric positive definite linear systems.

\subsection{Properties of gradient systems}
Given a differentiable function $f:\bbR^n\to\bbR$,
we consider the gradient systems of the form
\begin{align}\label{eq:odef}
\frac{\rmd}{\rmd t} \bx (t) = - P \nabla f(\bx (t)), \quad t>0
\end{align}
with an initial vector $\bx (0) = \bx_0\in\bbR^n$,
where 
$P\in\bbR^{n\times n}$ is an arbitrary symmetric positive definite matrix.\footnote{Usually,
the system of the form \eqref{eq:odef} is called a gradient system only when $P=I$;
however, in this paper we consider \eqref{eq:odef} as a gradient system as long as the matrix $P$ is symmetric 
positive definite.}
Several properties for gradient systems are known.
First, gradient systems given by \eqref{eq:odef} are dissipative in the sense that 
the function $f(\bx(t))$ is nonincreasing along the solution:
\begin{align}\label{c:dp}
\frac{\rmd}{\rmd t} f (\bx (t)) =  \nabla f (\bx (t)) ^\top \frac{\rmd}{\rmd t} \bx (t) =
-  \nabla f (\bx (t)) ^\top P \nabla f (\bx (t)) \leq 0.
\end{align}
Here, the first equality is just the chain rule.
The second equality follows from the 
substitution of \eqref{eq:odef}, 
and the last inequality follows from the symmetric positive definiteness of the matrix $P$.
Note that $\frac{\rmd}{\rmd t} f(\bx(t)) < 0$ unless $\nabla f = \bm{0}$.                                                       
We call 
the function $f$ `the energy function,' and the property \eqref{c:dp} `the energy dissipation property.'
Second, 
a clear relationship is observed between the optimisation problems and gradient systems
if we make additional assumptions on the energy function $f$.
To explain the relationship, we recall some definitions.

\begin{definition}
A function $f:\bbR^n\to\bbR$ is called 
\begin{itemize}
\item convex if and only if for all $\bx,\by\in\bbR^n$ and $\lambda\in [0,1]$,
\begin{align*}
f(\lambda \bx + (1-\lambda ) \by) \leq \lambda f(\bx) + (1-\lambda) f(\by).
\end{align*}
\item strictly convex if and only if for all $\bx,\by\in\bbR^n$ ($\bx\neq\by$) and $\lambda\in (0,1)$,
\begin{align*}
f(\lambda \bx + (1-\lambda ) \by) < \lambda f(\bx) + (1-\lambda) f(\by).
\end{align*}
\item coercive if and only if $f(\bx_n)\to\infty$ for $\| \bx_n\| \to \infty$.
\end{itemize}
\end{definition}

If the function $f$ is strictly convex and coercive, 
the flow $\bx(t)$ to the gradient system \eqref{eq:odef}  converges to the equilibrium
as $t\to\infty$ for any initial vector $\bx_0$
as an immediate consequence of the energy-dissipation property \eqref{c:dp}.
Note that the equilibrium of the gradient system \eqref{eq:odef} is the unique optimal solution to
the unconditioned optimisation problem
\begin{align}\label{eq:opt}
\min_{\bx\in\bbR^n} f(\bx),
\end{align}
or is the solution to $\nabla f(\bx)=\bm{0}$.
This discussion is summarised in the following proposition.

\begin{proposition}[Continuous case~\cite{hs13,sh96}] \label{thm:co}
Let a differentiable function $f:\bbR^n\to\bbR$ be strictly convex and coercive.
Then, the exact flow of the gradient system \eqref{eq:odef}
converges to the unique minimizer of the function $f$ for any initial vector $\bx_0$, 
that is,
\begin{align}
\lim _{t\to\infty} \bx (t) = \arg\min _{\bx\in\bbR^n} f(\bx).
\end{align}
\end{proposition}

\subsection{Gradient systems associated with symmetric positive definite linear systems}

To consider symmetric positive definite linear systems,\footnote{In this paper, we focus on symmetric positive definite linear systems just for simplicity; 
the discussion also applies for positive definite matrices.}
let us focus on the energy function
\begin{align}\label{eq:f}
f(\bx) = \frac12 \bx ^\top A \bx -\bx ^\top \bb 
\end{align}
for a given symmetric positive definite matrix $A\in\bbR^{n\times n}$ and vector $\bb\in\bbR^n$.
Note that this function $f$ is strictly convex and coercive;
therefore,
the minimiser of the optimization problem \eqref{eq:opt} is unique
and coincides with the exact solution to the linear system
\begin{equation} \label{eq:ls}
A \bx = \bb.
\end{equation}
Furthermore, the corresponding gradient system is written as
\begin{align}\label{eq:ode}
\frac{\rmd}{\rmd t} \bx (t) = - P \nabla f(\bx (t))
= -P (A\bx(t) - \bb).
\end{align}
It is clear from Proposition~\ref{thm:co} that as $t\to\infty$,
 the exact flow $\bx (t)$ converges to
the exact solution of the linear system \eqref{eq:ls}.

\section{New connection based on the discrete gradient method}
\label{sec3}
This section explains a new connection
between the stationary iterative methods with the discrete energy-dissipation 
property $f(\bx^{(k+1)}) \leq f(\bx^{(k)})$
and the gradient systems that always have the energy-dissipation property 
$\frac{\rmd}{\rmd t} f(\bx (t)) \leq 0$.
To make an intended connection, 
we need to consider numerical methods that inherit the energy-dissipation property
when applied to gradient systems.
The discrete gradient method possesses this property;
therefore,
we employ this method instead of the explicit Euler method
to understand the stationary iterative methods from a fresh perspective.

We review the discrete gradient method in the general setting in Subsection~\ref{subsec:dg}
and apply the discrete gradient method to the gradient system \eqref{eq:ode}
to show the resulting schemes that constitute a subclass of stationary iterative methods
with the discrete energy-dissipation 
property $f(\bx^{(k+1)}) \leq f(\bx^{(k)})$
 in Subsection~\ref{subsec:framework}.
The discrete gradient scheme coincided with the SOR method, and 
the equivalence is discussed in Subsection~\ref{subsec:sor}.
In Subsection~\ref{subsec:im}, we discuss the outcomes of the new connection.

\subsection{Discrete gradient method}
\label{subsec:dg}

The discrete gradient method for general gradient systems
is briefly reviewed.
Additional details are available in~\cite{go96,ia88,mq99,qm08,qt96}.\footnote{
In these references, the discrete gradient method has been developed to derive 
energy-preserving schemes for conservative systems, such as Hamiltonian systems.
However, it is also known that 
the method can be used to derive energy dissipative schemes
for gradient systems~\cite{fu99,hl14,mf14,mqr98}.
For a general description on numerical methods preserving some
underlying geometric properties of differential equations, see~\cite{hlw06}.}
The key idea is to devise the discretisation of the gradient.
We first define a
discrete gradient $\dg f$ as follows.

\begin{definition}[Discrete gradient~\cite{go96}]
Let $f:\bbR^n\to\bbR$ be continuously differentiable.
The function $\dg f: \bbR^n \times \bbR^n \to \bbR$
is called a discrete gradient if it satisfies
\begin{align}
f(\bx) - f(\by) &=  \dg f(\bx,\by)^\top (\bx-\by), \label{dgcond1} \\
\dg f(\bx,\bx) &= \nabla f(\bx) \label{dgcond2}
\end{align}
for all $\bx, \ \by \in \bbR^n$.
\end{definition}

The first condition, which is given by \eqref{dgcond1}, corresponds to the chain rule;
it is essential and is called the discrete chain rule.
The second condition, which is given by \eqref{dgcond2}, merely requires that the discrete gradient is 
an approximation to the gradient.
Note that $\dg f$ is a function 
that corresponds to the gradient as a function $\nabla f: \bbR^n \to \bbR^n$.

Let us leave aside the construction of concrete discrete gradients for the moment
and
assume that an appropriate discrete gradient has been found. 
Thereafter, for the gradient system \eqref{eq:odef}, the 
discrete gradient scheme is given as follows:
\begin{align} \label{dgscheme}
\frac{\bx^{(k+1)} - \bx ^{(k)}}{h} = - P \dg f(\bx^{(k+1)}, \bx^{(k)} ), \quad k=0,1,2,\dots.
\end{align}
For the solution to the scheme~\eqref{dgscheme},
the discrete energy-dissipation holds:
\begin{align}
\frac{1}{h} \left( f(\bx^{(k+1)}) - f(\bx^{(k)}) \right)
=   \dg f(\bx^{(k+1)}, \bx^{(k)} ) ^\top \frac{\bx^{(k+1)} - \bx ^{(k)}}{h}  
= - \dg f(\bx^{(k+1)}, \bx^{(k)} ) ^\top P \dg f(\bx^{(k+1)}, \bx^{(k)} ) \leq 0. \label{d:dp}
\end{align}
Here, the first equality is just the discrete chain rule;
the second equality follows from the substitution of the scheme~\eqref{dgscheme}.
The last inequality arises from the positive definiteness of the matrix $P$.
Hence, for the solution to the discrete gradient scheme \eqref{dgscheme},
the following property holds:

\begin{proposition}[Discrete case~\cite{gr17}] \label{prop:d}
Let a differentiable function $f:\bbR^n\to\bbR$ be strictly convex and coercive.
Then, the sequence $\{ \bx ^{(k)}\}_{k=0}^\infty$ obtained by the discrete gradient 
scheme given by~\eqref{dgscheme} converges to the unique minimiser of the function $f$ for any initial vector $\bx_0$,
that is,
\begin{align}
\lim _{k\to\infty} \bx^{(k)} = \arg\min _{\bx\in\bbR^n} f(\bx).
\end{align}
\end{proposition}

A rigorous proof for this is given in~\cite{gr17}.
This property is a discrete counterpart of Proposition~\ref{thm:co},
which indicates that any discrete gradient scheme for gradient systems with a strictly 
convex and coercive function $f$ always converges to the equilibrium point independently of the stepsize $h$.
As a corollary, discrete gradient schemes for the gradient system of the form \eqref{eq:ode}
are iterative algorithms for the linear system \eqref{eq:ls} with unconditional convergence properties
in terms of the choice of $h$.

Let us reconsider the construction of discrete gradients.
In general, this construction is not unique.
Below,
we list three approaches.

\begin{itemize}
\item Itoh and Abe~\cite{ia88} defined the discrete gradient as follows:
\begin{align}\def\arraystretch{1.3}
\dg f(\bx,\by) = 
\begin{bmatrix}
\cfrac{f(x_1,y_2,\dots,y_n)-f(y_1,\dots, y_n)}{x_1-y_1} \\
\cfrac{f(x_1,x_2,y_3\dots,y_n)-f(x_1,y_2,\dots, y_n)}{x_2-y_2} \\
\vdots\\
\cfrac{f(x_1,\dots,x_n)-f(x_1,\dots,x_{n-1}, y_n)}{x_n-y_n}
\end{bmatrix} . \label{eq:ia}
\end{align}
This approach is derivative free; therefore, the resulting discrete gradient scheme is computationally 
cheaper than others.
We note that the resulting scheme has only first order accuracy.\footnote{
An ODE solver is said to be of order $p$
if $\| \bx^{(1)}  - \bx(h) \| = O (h^{p+1})$ holds.
}
\item Gonzalez~\cite{go96} defined the discrete gradient as follows:
\begin{align} \label{gdg}
\dg f(\bx,\by) = \nabla f \left( \frac{\bx+\by}{2} \right) +
\frac{f(\bx)-f(\by) -\nabla f \left( \frac{\bx+\by}{2} \right) ^\top (\bx-\by) }{\| \bx-\by\| ^2} (\bx-\by).
\end{align}
This discrete gradient is a modification of the midpoint rule.
The resulting scheme is second order, but the implementation could be cumbersome because of
the term $\| \bx-\by\|^2$ appearing in the denominator.
\item The average vector field discrete gradient~\cite{hl83,mq99,qm08} is given
as follows:
\begin{align}\label{avf}
\dg f(\bx,\by) = 
\int_0^1 \nabla f (\xi \bx + (1-\xi) \by)\,\mathrm{d}\xi.
\end{align}
This approach is also second order, and 
we do not need to know the function $f$.
The resulting scheme is fully implicit, although the implementation effort is usually  relaxed 
compared with Gonzalez's approach.
\end{itemize}
For other approaches, see~\cite{mq99}
and the references therein.

\subsection{Discrete gradient schemes for gradient systems}
\label{subsec:framework}
According to Proposition~\ref{prop:d},
the sequence $\{ \bx ^{(k)} \}_{k=0}^\infty$ of every discrete gradient scheme 
converges to the exact solution of linear systems
independently of the stepsize $h$ (as long as $h>0$).
Therefore, discrete gradient schemes constitute a certain subclass of stationary iterative methods
with the discrete energy-dissipation property $f(\bx^{(k+1)} ) \leq f(\bx ^{(k)})$.
However, it is not clear whether all stationery iterative methods with the energy-dissipation property
can be regarded as discrete gradient schemes;
this concern will be addressed in our future work.
In this subsection, we determine whether
the subclass embraces practical and important stationary iterative methods.
Therefore,
we need to 
obtain concrete discrete gradients for the energy function \eqref{eq:f}
to formulate concrete discrete gradient schemes.

For the energy function \eqref{eq:f},
Gonzalez's method \eqref{gdg} and the average vector field method \eqref{avf} give the same
discrete gradient $\dg f (\bx,\by) = A (\bx + \by)/2 - \bb$.
The corresponding discrete gradient scheme is given as follows:
\begin{align} \label{midpoint}
\frac{\bx^{(k+1)} - \bx ^{(k)}}{h} = -P\left(A \frac{\bx^{(k+1)} +\bx^{(k)}}{2} - \bb \right).
\end{align}
A linear system with the coefficient matrix $I+hPA/2$ should be solved at each time step;
therefore, \eqref{midpoint} is not practical in general (see~\cite{ch88} 
for the connection to the Pad\'e approximation).

Another possibility is to employ the Itoh--Abe discrete gradient \eqref{eq:ia}.
In this case,
the $i$th component of the discrete gradient is calculated as
\begin{align}
(\dg f(\bx,\by))_i &= \frac{f(x_1,\dots,x_i,y_{i+1},\dots,y_n)-f(x_1,\dots,x_{i-1},y_i,\dots, y_n)}{x_i-y_i} \\
&= \sum_{j<i} a_{ij}x_j + a_{ii}\frac{x_i+y_i}{2} + \sum_{j>i} a_{ij}y_j - b_i,
\end{align}
where $\sum_{j<i}$ and $\sum_{j>i}$ are the abbreviations of 
$\sum_{j=1}^{i-1}$ and $\sum_{j=i+1}^n$, respectively.
Thereafter, for the most natural choice $P=I$, the discrete gradient scheme
\eqref{dgscheme} is written
in a component-wise fashion as
\begin{align}
\frac{x_i^{(k+1)} - x_i^{(k)}}{h} &= 
- (\dg f(\bx^{(k+1)}, \bx^{(k)} ) )_i  \\
&=
-\left( \sum_{j<i} a_{ij}x_j^{(k+1)} + a_{ii}\frac{x_i^{(k+1)}+x_i^{(k)}}{2} + \sum_{j>i} a_{ij}x_j^{(k)} - b_i \right),
\end{align}
which can also be expressed as
\begin{align} \label{ourscheme1}
x_i^{(k+1)} = \frac{1}{1+\frac{h}{2}a_{ii}}
\left[ -h\sum_{j<i} a_{ij}x_j^{(k+1)} + \left( 1-\frac{h}{2}a_{ii}\right) x_i^{(k)} - h \sum_{j>i} a_{ij}x_j^{(k)} + hb_i \right].
\end{align}
In this case, 
a linear system does not need to be solved for each iteration.

Let us also consider the case when $P = D^{-1}$, where $D:=\diag (a_{11}, a_{22},\dots,a_{nn})$.
Note that
the symmetric positive definiteness of the matrix $A$ indicates $a_{ii}>0$;
therefore,
$P$ is (symmetric) positive definite. In this case,
the Itoh--Abe discrete gradient scheme is given by
\begin{align} 
x_i^{(k+1)} 
 = \frac{1}{1+\frac{h}{2}}
\left[ -ha_{ii}^{-1}\sum_{j<i} a_{ij}x_j^{(k+1)} + \left( 1-\frac{h}{2}\right) x_i^{(k)} 
 - h a_{ii}^{-1}\sum_{j>i} a_{ij}x_j^{(k)} + ha_{ii}^{-1} b_i \right]. \label{ourscheme1d}
\end{align}

\subsection{Equivalence between the scheme given by {\rm (\ref{ourscheme1d})} and the SOR method}
\label{subsec:sor}

The Itoh--Abe discrete gradient scheme \eqref{ourscheme1d} is equivalent to the SOR method.
This subsection explains the equivalence.

Let us express matrix $A$ as the matrix sum
\begin{align}\label{DLU}
A = D + L + U,
\end{align}
where $D=\diag (a_{11}, a_{22},\dots,a_{nn})$,
and $L$ and $U$ are strictly lower and upper triangular $n\times n$ matrices, respectively.
The Itoh--Abe discrete gradient scheme \eqref{ourscheme1d} is a stationary iterative method of the form
\begin{align} \label{eq:sor}
\bx^{(k+1)}= G\bx^{(k)} + \bc,
\end{align}
where $G$ and $\bc$ are expressed as
\begin{align}
G_{\mathrm{DG}}&= \left[ \left( I + \frac{h}{2}\right) D + hL\right]^{-1} \left[\left( I-\frac{h}{2}\right)D-hU\right],  \label{eq:m}\\
\bc_{\mathrm{DG}} &= h \left[ \left( I + \frac{h}{2}\right)D + hL\right]^{-1}  \bb \label{eq:c}.
\end{align}
For the SOR method, the iterative matrix and vector can be written as
\begin{align}
G_\sor&= \left( D + \omega L\right)^{-1} \left[(1-\omega)D-\omega U\right], \\
\bc_\sor &= \omega \left( D + \omega L\right)^{-1}  \bb.
\end{align} 
This proves that the scheme \eqref{ourscheme1d}
is equivalent to the SOR method
in the following sense:

\begin{theorem}
\label{th:eq}
The Itoh--Abe discrete gradient scheme \eqref{ourscheme1d} and the SOR method are equivalent
in the sense that $G_{\mathrm{DG}}=G_\sor$ and $\bc_{\mathrm{DG}} = \bc_\sor$
if the relationship between the two parameters is given as follows:
\begin{align}\label{condh}
h = \frac{2\omega}{2-\omega}.
\end{align}
\end{theorem}

\begin{proof}
$G_{\mathrm{DG}}=G_\sor$ and $\bc_{\mathrm{DG}} = \bc_\sor$ are checked by substituting
\eqref{condh} into \eqref{eq:m} and \eqref{eq:c}.
\end{proof}

\subsection{Discussion}
\label{subsec:im}

We discuss the consequences of the new connection between
the gradient systems and
the
stationary iterative methods with the discrete energy-dissipation property,
by contrasting them with the standard connection.
The new connection overcomes the shortcomings mentioned in Section~\ref{sec1}.

\begin{enumerate}
\item[(a)]
The new connection offers some new insights, particularly for the SOR method.
It is well known that the convergence condition of the SOR method in terms of the relaxation
parameter $\omega$ is $\omega \in (0,2)$~\cite{go13,va00}.
Theorem~\ref{th:eq} implies that the SOR method with $0<\omega<2$ can be 
regarded as a stationary iterative method with the unconditional convergence property
in terms of the stepsize $h$.
It is now clear that the SOR method has the energy-dissipation property,
and the SOR method can be derived on the basis of the energy-dissipation
principle,
in contrast to the standard introduction of the SOR method as an extension of 
the Gauss--Seidel method.
The energy-dissipation is followed even for the update of each component:
\begin{align*}
& f(x_1^{(k+1)}, \dots, x_i^{(k+1)}, x_{i+1}^{(k)},\dots, x_n^{(k)} ) 
- f(x_1^{(k+1)}, \dots, x_{i-1}^{(k+1)}, x_i^{(k)},\dots, x_n^{(k)} ) \\
& \quad =
-\frac{1}{a_{ii}} \frac{h}{(1+\frac{h}{2})^2}
\left( \sum_{j<i} a_{ij} x_j^{(k+1)} + a_{ii} x_i^{(k)} + \sum_{j>i} a_{ij}x_j^{(k)} - b_i \right)^2 \leq 0.
\end{align*}
If the stepsize $h$ is chosen as $h=2$ or equivalently as $\omega=1$ 
because of
\eqref{condh},
then the decrement is locally maximised; this case coincides with the Gauss--Seidel method.

\item[(b)] The discrete gradient schemes converge to the solution of linear systems,
which will be exemplified in the next section.
Furthermore, one can freely choose any of the symmetric positive definite matrices $P$,
in contrast to the standard connection based on the explicit Euler method.
As long as the matrix $P$ is symmetric positive definite and 
the discrete gradient method is being used,
convergence is always guaranteed for 
the resulting discrete gradient schemes.
\end{enumerate}

\section{SOR-type methods}
\label{sec4}

We know that discrete gradients are not unique.
We provide two examples for this.
We show that the symmetric SOR method and 
block SOR method are equivalent to certain discrete gradient schemes.

\subsection{Symmetric SOR method}

We observe that the following function
\begin{align}
\nabla _\rmd f(\bx,\by) = 
\begin{bmatrix}
\cfrac{f(x_1,x_2,\dots,x_n)-f(y_1,x_2,\dots, x_n)}{x_1-y_1} \\
\vdots\\
\cfrac{f(y_1,\dots,y_{n-2},x_{n-1},x_n)-f(y_1,\dots,y_{n-2},y_{n-1}, x_n)}{x_n-y_n} \\
\cfrac{f(y_1,\dots,y_{n-1},x_n)-f(y_1,\dots,y_{n-1}, y_n)}{x_n-y_n}
\end{bmatrix} 
\end{align}
satisfies conditions \eqref{dgcond1} and \eqref{dgcond2}.
This function is a variant of the Itoh--Abe discrete gradient.
In particular, when $P=D^{-1}$, the corresponding discrete gradient scheme
reads as follows:
\begin{align} 
x_i^{(k+1)}  
=\frac{1}{1+\frac{h}{2}}
\left[ -ha_{ii}^{-1}\sum_{j<i} a_{ij}x_j^{(k)} + \left( 1-\frac{h}{2}\right) x_i^{(k)} 
  - ha_{ii}^{-1} \sum_{j>i} a_{ij}x_j^{(k+1)} + ha_{ii}^{-1}b_i \right].\label{ourscheme2} 
\end{align}
The scheme that uses \eqref{ourscheme1d} and \eqref{ourscheme2}  alternately
reduces to a symmetric solver, which is equivalent to the symmetric SOR method~\cite[Chapter 11.2]{go13}
with the relationship \eqref{condh}.
In this sense, the symmetric SOR method is connected to the gradient system \eqref{eq:ode} with $P=D^{-1}$.

\subsection{Block SOR method}
In this subsection, we show that the block SOR method~\cite{ar56,sa03,va00} is also equivalent to a certain discrete gradient scheme.
Let us suppose that $A$ is in the following $p\times p$ block partitioned form:
\begin{align}
A = \begin{bmatrix}
A_{11} & A_{12} & \cdots  & A_{1p} \\
A_{21} & A_{22} & \cdots & A_{2p} \\
\vdots & \vdots & \ddots & \vdots \\
A_{p1} & A_{p2} & \cdots & A_{pp} 
\end{bmatrix}.
\end{align}
Correspondingly, the vector $\bx$ is also divided into $p$ parts: $\bx = (\bx_1 ^\top, \bx_2 ^\top, \dots , \bx_p ^\top)^\top$.

We now extend the pointwise Itoh--Abe discrete gradient to a blockwise fashion.
We note that the $i$th component of the Itoh--Abe discrete gradient \eqref{eq:ia}
can be written as
\begin{align*}
\int_0^1 \partial _i f(x_1,\dots, x_{i-1}, \xi x_i+ (1-\xi) y_i, y_{i+1},\dots, y_n) \,\rmd \xi,
\end{align*}
because
\begin{align*}
&f(x_1,\dots, x_i, y_{i+1},\dots, y_n) - f(x_1,\dots, x_{i-1}, y_i,\dots, y_n) \\
&\quad =
\int_0^1 \frac{\rmd}{\rmd \xi} f(x_1,\dots, x_{i-1}, \xi x_i+ (1-\xi) y_i, y_{i+1},\dots, y_n) \,\rmd \xi \\
&\quad =
\int_0^1 \partial _i f(x_1,\dots, x_{i-1}, \xi x_i+ (1-\xi) y_i, y_{i+1},\dots, y_n) \cdot (x_i-y_i) \,\rmd \xi,
\end{align*}
where $\partial _i$ denotes the partial derivative with respect to the $i$th variable.
This observation motivates us to consider the following as a possible new discrete gradient:
\begin{align} 
 (\dg f(\bx,\by))_i 
 = \int_0^1 \nabla _i f(\bx_1,\dots, \bx_{i-1}, \xi \bx_i+ (1-\xi) \by_i, \by_{i+1},\dots, \by_p) \, \rmd \xi, \quad i= 1,\dots ,p, \label{eq:iablock} 
\end{align}
where the blockwise notation is employed;
the subscript $i$ denotes the $i$th block (not the $i$th component), and $\nabla _i$ denotes the gradient with respect to the $i$th block.

\begin{proposition}
Function \eqref{eq:iablock} falls into the category of discrete gradients
in the sense that
it satisfies \eqref{dgcond1} and \eqref{dgcond2}.
\end{proposition}

\begin{proof}
Since
\begin{align*}
&f(\bx_1,\dots, \bx_i, \by_{i+1},\dots, \by_p) - f(\bx_1,\dots, \bx_{i-1}, \by_i,\dots, \by_p) \\
&\quad =
\int_0^1 \frac{\rmd}{\rmd \xi} f(\bx_1,\dots, \bx_{i-1}, \xi \bx_i+ (1-\xi) \by_i, \by_{i+1},\dots, \by_p) \,\rmd \xi \\
&\quad =
\int_0^1 \nabla _i f(\bx_1,\dots, \bx_{i-1}, \xi \bx_i+ (1-\xi) \by_i, \by_{i+1},\dots, \by_p) ^\top (\bx_i-\by_i) \,\rmd \xi,
\end{align*} 
condition \eqref{dgcond1} is satisfied.
Condition \eqref{dgcond2} is checked as follows:
\begin{align} 
(\dg f(\bx,\bx))_i = \int_0^1 \nabla _i f(\bx_1,\dots, \bx_{i-1},  \bx_i, \bx_{i+1},\dots, \bx_p) \, \rmd \xi
= (\nabla f(\bx))_i.
\end{align}
\end{proof}

For the energy function \eqref{eq:f},
we can calculate the new discrete gradient by substituting \eqref{eq:f} into 
\eqref{eq:iablock}:
\begin{align}
(\dg f(\bx,\by))_i = \sum_{j<i} A_{ij} \bx_j + A_{ii}\frac{\bx_i+\by_i}{2} + \sum_{j>i} A_{ij}\by_j - \bb_i,
\quad i=1,\dots,p.
\end{align}
The corresponding discrete gradient scheme for the case $P = D_\mathrm{b}^{-1}$, 
where $D_\mathrm{b} := \diag (A_{11}, \dots, A_{pp})$ is a block diagonal matrix,
is given by
\begin{align}
\frac{\bx_i^{(k+1)} - \bx_i^{(k)}}{h} 
=
-A_{ii}^{-1}\left( \sum_{j<i} A_{ij}\bx_j^{(k+1)} + A_{ii}\frac{\bx_i^{(k+1)}+\bx_i^{(k)}}{2} + \sum_{j>i} A_{ij}\bx_j^{(k)} - \bb_i \right),
\end{align}
or equivalently by
\begin{align} 
\bx_i^{(k+1)} 
 = \frac{1}{1+\frac{h}{2}}
\left[ -hA_{ii}^{-1}\sum_{j<i} A_{ij}\bx_j^{(k+1)} + \left( I-\frac{h}{2}\right) \bx_i^{(k)} 
 - h A_{ii}^{-1}\sum_{j>i} A_{ij}\bx_j^{(k)} + hA_{ii}^{-1}\bb_i \right].\label{ourscheme3}
\end{align}
Let $A = D_\mathrm{b} + L_\mathrm{b} + U_\mathrm{b}$,
where $L_\mathrm{b} $ and $U_\mathrm{b}$ are strictly lower and upper block triangular 
$n\times n$ matrices, respectively.
The iteration scheme given by \eqref{ourscheme3} is a stationary iterative method of  the form
\begin{align} 
\bx^{(k+1)}= G\bx^{(k)} + \bc,
\end{align}
where $G$ and $\bc$ are expressed as
\begin{align}
G_{\mathrm{BDG}}&= \left[ \left( I + \frac{h}{2}\right) D_\mathrm{b} + hL_\mathrm{b}\right]^{-1} \left[\left( I-\frac{h}{2}\right)D_\mathrm{b}-hU_\mathrm{b}\right],  \label{eq:bm}\\
\bc_{\mathrm{BDG}} &= h \left[ \left( I + \frac{h}{2}\right)D_\mathrm{b} + hL_\mathrm{b}\right]^{-1}  \bb \label{eq:bc}.
\end{align}
For the block SOR method, the iterative matrix and vector can be written as
\begin{align}
G_{\mathrm{BSOR}}&= \left( D_\mathrm{b} + \omega L_\mathrm{b}\right)^{-1} \left[(1-\omega)D_\mathrm{b}-\omega U_\mathrm{b}\right], \\
\bc_{\mathrm{BSOR}} &= \omega \left( D_\mathrm{b} + \omega L_\mathrm{b}\right)^{-1}  \bb.
\end{align} 

It is proved that the discrete gradient scheme \eqref{ourscheme3} is equivalent to the block SOR method
in the following sense.

\begin{proposition}
The scheme \eqref{ourscheme3} and the block SOR method are equivalent
in the sense that $G_{\mathrm{BDG}}=G_{\mathrm{BSOR}}$ and $\bc_{\mathrm{BDG}} = \bc_{\mathrm{BSOR}}$
if the relationship between the two parameters is given by \eqref{condh}.
\end{proposition}

The proof for this proposition is similar to that of Theorem~\ref{th:eq}.

\section{Concluding remarks}
\label{sec:cr}

In this paper, 
we described the connection between stationary iterative methods with 
the discrete energy-dissipation property $f(\bx^{(k+1)} )\leq f(\bx^{(k)})$ and gradient systems (see Fig.~\ref{fig:newconnection}).
The focus of the discussion was the equivalence between the SOR-type methods and the discrete gradient schemes of the Itoh--Abe type.
It is still not clear whether all stationary iterative methods with the discrete energy-dissipation property
can be explained by the connection (which is why we used the dashed line in Fig.~\ref{fig:newconnection}); 
the connection can explain some famous methods such as the SOR, symmetric SOR and block SOR methods.
We found a new way to derive
these methods; these methods monotonically decrease the energy function \eqref{eq:f}, and 
the relaxation parameter $\omega$ is related to the stepsize \eqref{condh}.
It is hoped that the connection discussed in this paper sheds new light on the study of stationary 
iterative methods.
We obtained \eqref{eq:iablock}, the new discrete gradient, while studying the block
SOR method. Discrete gradients similar to \eqref{eq:iablock} will be tested for other gradient systems and nonlinear optimisation problems.
\begin{figure}[htbp]
\centering
\begin{tikzpicture}[
box0/.style={
	draw=white,
	text width=12em,
	minimum height=5em,
	rectangle split,
	rectangle split parts=3,
	rectangle split draw splits=false,
	},
box1/.style={
	rectangle,
	dashed,
	rounded corners,
	draw=black, very thick,
	text width=12em,
	minimum height=5em,
	},
box2/.style={
	rectangle,
	rounded corners,
	draw=black, very thick,
	text width=15em,
	minimum height=5em,
	},
box3/.style={
	draw=white,
	text width=12em,
	minimum height=5em,
	rectangle split,
	rectangle split parts=3,
	rectangle split draw splits=false,
	},
]
\node[box0] (l1) {$\cdot$\ SOR method\nodepart{two} $\cdot$\  symmetric SOR method \nodepart{three} $\cdot$\  block SOR method};
\node[box1, ultra thick, fit={(l1)($(l1.north)+(0,0.5em)$)($(l1.east)+(0.5em,0)$)($(l1.west)+(-0.5em,0)$)($(l1.south)+(0,-0.5em)$)},align=left] (l2) {};
\node[box2,ultra thick, fit={(l2)($(l2.north)+(0,1.5em)$)($(l2.east)+(0.5em,0)$)($(l2.west)+(-0.5em,0)$)($(l2.south)+(0,-0.5em)$)},align=left] (l3) {};
\node[below right, inner sep=5pt] at (l3.north west) {$\{ G \in \bbR^{n\times n} \ | \ f(\bx^{(k+1)}) \leq f(\bx^{(k)}) \}$};
\node[box2, fit={(l1)($(l3.north)+(0,1.5em)$)($(l3.east)+(1em,0)$)($(l3.west)+(-0.5em,0)$)($(l3.south)+(0,-0.5em)$)},align=left] (l4) {};
\node[below right, inner sep=5pt] at (l4.north west) {$\{ G \in \bbR^{n\times n} \ | \ G = M^{-1} N \ (A=M-N) \}$};
\node[above= 0.7cm of l4] () {Stationary iterative methods}; 
\node[above= 0.1cm of l4] () {($\bx^{(k+1)} = G \bx^{(k)} + \bc$)}; 
\node[box2,minimum height=6em,ultra thick, right=4cm of l1,align=left] (r1) {};
\node[below right, inner sep=5pt,	rectangle split,
	rectangle split parts=3,] at (r1.north west) { \leftline{$\{P  \in \bbR^{n\times n} \ | \ \frac{\rmd}{\rmd t } f(\bx) \leq 0 \}$} };
\node[box2, fit={(r1)($(r1.north)+(0,3.5em)$)($(r1.west)+(-0.5em,0)$)($(r1.east)+(0.5em,0)$)($(r1.south)+(0,-1.5em)$)},align=left] (r3) {};
\node[below right, inner sep=5pt] at (r3.north west) {$\{ P \in \bbR^{n\times n} \ | \ \det P \neq 0  \}$ };
\node[above= 0.7cm of r3] () {Continuous dynamical systems}; 
\node[above= 0.1cm of r3] () {($\dot{\bx} = -P (A\bx - \bb)$)}; 
\path[<->,ultra thick] (l2.east) edge node [above=0.4cm] {\hspace{1em} DG method} node [above=0.03cm] {\hspace{1em}  ($h>0$)} (r1.west) + (0,1em);
\end{tikzpicture}
\caption{Connection between gradient systems and stationary iterative methods with the energy-dissipation property
$f(\bx^{(k+1)} )\leq f(\bx^{(k)})$.
`DG' has been used as an abbreviation for discrete gradient.
The set $\{P  \in \bbR^{n\times n} \ | \ \frac{\rmd}{\rmd t } f(\bx) \leq 0 \}$ coincides with 
the set of positive definite matrices of size $n$-by-$n$.
}
\label{fig:newconnection}
\end{figure}

\bibliographystyle{plain}
\bibliography{references}
\end{document}